\documentclass[12pt, twoside, ]{article}
\usepackage{amsfonts}
\usepackage{amsmath,amsthm}
\usepackage{enumerate}

\usepackage[OT1]{fontenc}
\usepackage{bm}
\usepackage{mathrsfs}
\allowdisplaybreaks[4]

\newtheorem{thm}{Theorem}[section]
\newtheorem{lem}[thm]{Lemma}
\newtheorem{rem}[thm]{Remark}
\newtheorem{pro}[thm]{Proposition}

\newtheorem{cor}[thm]{Corollary}
\numberwithin{equation}{section}

\begin{document}

\title{\textbf{Mertens' theorem and prime number theorem for Selberg class}}
\author{Yoshikatsu Yashiro\\ 
\small Graduate School of Mathematics, Nagoya University,\\[-4.8pt] 
\small 464-8602 \ Chikusa-ku, Nagoya, Japan \\[-4.8pt] 
\small E-mail: m09050b@math.nagoya-u.ac.jp
}
\date{}
\maketitle

\renewcommand{\thefootnote}{}

\quad\\[-60pt]

\begin{abstract}
In 1874, Mertens proved the approximate formula for partial Euler product for Riemann zeta function at $s=1$, which is called Mertens' theorem. 
In this paper, we shall generalize Mertens' theorem for Selberg class and show the prime number theorem for Selberg class.  
\end{abstract}

\makeatletter
\def\Res{\mathop{\operator@font Res}}
\makeatother

\section{Introduction}

In 1874, Mertens \cite{MER} proved the following theorem:
\begin{align*}
\prod_{p\leq x}\left(1-\frac{1}{p}\right)^{-1}=e^\gamma\log x+O(1),
\end{align*}
where $x\in\mathbb{R}_{\geq2}$ and $\gamma$ is Euler's constant. The above formula is the approximate formula of the finite Euler product for the Riemann $\zeta$-function $\zeta(s)$ at $s=1$, which is called Mertens' (3rd) theorem. 
Later, in 1999 Rosen \cite{ROS} generalized Mertens' theorem for Dedekind $\zeta$-function $\zeta_K(s)$:
\begin{align*}
\prod_{N\mathfrak{p}\leq x}\left(1-\frac{1}{N\mathfrak{p}}\right)^{-1}=\alpha_Ke^{\gamma}\log x+O(1),
\end{align*}
where $K$ is an algebraic number field, and $\alpha_K$ is the residue of $\zeta_K(s)$ at $s=1$. This theorem is obtained by using   
the following approximate formula:
\begin{align}
\sum_{N\mathfrak{p}\leq x}\log(N\mathfrak{p})=x+O(xe^{-c_K\sqrt{\log x}}) \label{CCC}
\end{align}
which was proved by Landau \cite{LAN}. Note that (\ref{CCC}) is equivalent to the prime number theorem for the algebraic number field $K$, where $c_K$ is a positive constant depending on $K$.

In this paper, we consider Mertens' theorem for Dirichlet series introduced by Selberg \cite{SE2}.  Selberg class $\mathcal{S}$ is defined by the class of Dirichlet series satisfying the following conditions:
\begin{enumerate}[(a)]
\item (Absolute convergence) \label{CV} The series $F(s)=\sum_{n=1}^\infty a_F(n)n^{-s}$ is absolutely convergent for $\text{Re }s>1$.
\item (Analytic continuation) \label{AC} There exists  $m\in\mathbb{Z}_{\geq 0}$ such that $(s-1)^mF(s)$ is an entire function of finite order.
\item (Functional equation) \label{FE} The function $F(s)$ satisfies $\Phi(s)=\omega\overline{\Phi(1-\overline{s})}$ where 
$\Phi(s)=Q^s\prod_{j=1}^r\Gamma(\lambda_js+\mu_j)F(s)$, $r\in\mathbb{Z}_{\geq1}$, $Q\in\mathbb{R}_{>0}$, $\lambda_j\in\mathbb{R}_{>0}$, $\text{Re }\mu_j\in\mathbb{R}_{>0}$, and $\omega\in\mathbb{C}$ satisfying $|\omega|=1$.
\item (Ramanujan conjecture) \label{RC} For any fixed $\varepsilon\in\mathbb{R}_{>0}$, $a_F(n)=O(n^\varepsilon)$.
\item \label{LG} The logarithmic function of $F(s)$ is given by $\log F(s)=\sum_{n=1}^\infty b_F(n)n^{-s}$,   where $b_F(n)=0$ when $n\ne p^r \ (r\in\mathbb{Z}_{\geq1})$, and there exists $\theta\in\mathbb{R}_{<1/2}$ such that $b_F(n)=O(n^\theta)$. 
\end{enumerate}
Moreover, the extended Selberg class $\mathcal{S}^{\#}$ is defined by a class of Dirichlet series satisftying only the conditions (\ref{CV})--(\ref{FE}). Clearly $\mathcal{S}\subset \mathcal{S}^{\#}$. For example $\zeta(s)$ belongs to $\mathcal{S}$ and $\zeta_K(s)$ belong to $\mathcal{S}^{\#}$.
The function $\zeta(s)$ and $\zeta_K(s)$ have the Euler products and zero-free regions. It is expected that $F\in\mathcal{S}^{\#}$ satisfies the following conditions:
\begin{enumerate}[(I)]
\item (Euler product) \label{AEP} There exists  a positive integer $k$ (depending on $F$) such that 
\begin{align}
F(s)=\prod_{p}\sum_{r=0}^\infty\frac{a_F(p^r)}{p^{rs}}=\prod_{p}\prod_{j=1}^k\left(1-\frac{\alpha_j(p)}{p^s}\right)^{-1} \label{EP}
\end{align}
with $|\alpha_j(p)|\leq 1$ for $\text{Re }s>1$. 
\item (Zero-free region) \label{AZF} There exists a positive constant $c_F$ (depending on $F$) such that $F(s)$ has no zeros in the region $$\text{Re }s\geq 1-\frac{c_F}{\log(|\text{Im }s|+2)},$$ except $s=1$ (if $F(s)$ has zero at $s=1$) and the Siegel zeros of $F(s)$.
\end{enumerate}
\begin{rem}\label{NEP}
If $F\in \mathcal{S}^{\#}$ satisfies (\ref{AEP}), then (\ref{LG}) are satisfied with $b_F(p^r)=(\alpha_1(p)^r+\cdots+\alpha_k(p)^r)/r$ and $\theta=0$, where the constant of $O$-term depends on $k$.
\end{rem}
By the same discussion as in the proof of the zero-free region of $\zeta(s)$ (see Montgomery and Vaughan \cite[Lemma 6.5 and Theorem 6.6]{M&V}), the following fact is obtained:
\begin{itemize}
\setlength{\itemsep}{24pt}
\item[]\label{NZF} If $F\in \mathcal{S}^{\#}$ satisfies (\ref{LG}) with ${\rm Re \ }b_F(n)\geq 0$ and has a zero or a simple pole at $s=1$, then $F$ satisfies (\ref{AZF}). 
\end{itemize}

Applying Motohashi's method \cite[Chapter 1.5]{MOT}, we can extend Mertens' theorem for Selberg class by using Perron's formula and complex analysis. 
\begin{thm}[Mertens' 3rd theorem for Selberg class]\label{MTH1}
Let $F\in\mathcal{S}^{\#}$ and suppose the condition (\ref{AEP}) and (\ref{AZF}). Then we have
\begin{align}
\prod_{p\leq x}\prod_{j=1}^k\left(1-\frac{\alpha_j(p)}{p}\right)^{-1}=c_{-m}e^{\gamma m}(\log x)^{m}(1+O(e^{-C_F\sqrt{\log x}})), \label{MT1}
\end{align}
where $m$ denotes the order $m$ of pole for $F(s)$ at $s=1$ when $m\in\mathbb{Z}_{>0}$, and the order $-m$ of zero for $F(s)$ at $s=1$ when $m\in\mathbb{Z}_{\leq 0}$. Moreover $c_{-m}$ is given by 
$c_{-m}=\lim_{s\to1}(s-1)^mF(s)$, and $C_F$ is a positive constant smaller than $c_F$ of the condition (\ref{AZF}). 
\end{thm}
Indeed instead of the condition (\ref{AZF}), we can prove the following weaker formula than (\ref{MT1}):
\begin{align}
\prod_{p\leq x}\prod_{j=1}^k\left(1-\frac{\alpha_j(p)}{p}\right)^{-1}=c_{-m}e^{\gamma m}(\log x)^{m}\left(1+O\left(\frac{1}{\log x}\right)\right), \label{MT2}
\end{align}
under the assumption of the prime number theorem for Selberg class. In order to improve the error term in (\ref{MT2}), it is necessary to assume (\ref{AZF}).

It is well-known that the prime number theorem is equivalence to $\zeta(1+it)\ne0 \ (t\in\mathbb{R})$. Kaczorowski and Perelli \cite{KP1} proved the equivalence of prime number theorem in Selberg class: 
\begin{align}
F(1+it)\ne0 \ (t\in\mathbb{R}) \ \Longleftrightarrow \ \sum_{n\leq x}b_F(n)\log n=mx+o(x) \label{PNT1}
\end{align}
where $F\in\mathcal{S}$. If we apply (\ref{MT1}), we can improve the above $o(x)$:
\begin{thm}[Prime number theorem for Selberg class]\label{MTH2}
Let $F\in\mathcal{S}^{\#}$ and suppose the conditions (\ref{AEP}) and (\ref{AZF}). Then we have
\begin{align}
\sum_{n\leq x}b_F(n)\log n=mx+O(xe^{-C_F''\sqrt{\log x}}) \label{PNT2}
\end{align}
where $C_F''$ is a positive constant smaller than $C_F$ in Theorem \ref{MTH1}. 
\end{thm}
We shall give an example of Theorems \ref{MTH1} and \ref{MTH2}. In the case of $\zeta_K\in\mathcal{S}^{\#}$, we know that $\zeta_K$ satisfies (\ref{AEP}), (\ref{AZF}), and it is known that $\zeta_K(s)$ has a simple pole as $s=1$. Therefore the following fact is obtained:
\begin{cor} We obtain the Metens' theorem for $\zeta_K(s)$: 
\begin{align*}
\prod_{N\mathfrak{p}\leq x}\left(1-\frac{1}{N\mathfrak{p}}\right)^{-1}=\alpha_Ke^{\gamma}(\log x)(1+O(e^{-C_K\sqrt{\log x}}))
\end{align*}
and the prime number theorem for $\zeta_K(s)$:
\begin{align*}
\sum_{N{\mathfrak{p}}^r\leq x}\log(N\mathfrak{p}^r)=x+O(xe^{-c_K\sqrt{\log x}}),
\end{align*}
where $\alpha_K$ is the residue for $\zeta_K(s)$ at $s=1$, and $C_K, c_K$ are  positive constants such that $c_K<C_K$ depending on $K$.
\end{cor}
In the case of automorphic $L$-function, we see that the Rankin-Selberg $L$-function $L_{f\times g}(s)$ belongs to the Selberg class, and it is known that if $L_{f\times g}(s)$ has a simple pole at $s=1$ when $f=g$ and no pole in the whole $s$-plane when $f\ne g$, where $f$ and $g$ are cusp forms of weight $k$ for $SL_2(\mathbb{Z})$. Assume that $f, g$ are normalized Hecke eigenforms, and the Fourier expansion of $X=f, g$ are given by $X(z)=\sum_{n=1}^\infty\lambda_X(n)n^{(k-1)/2}e^{2\pi inz}$. Then we obtain the following corollary: 
\begin{cor} We have the Metens' theorem for $L_{f\times g}(s)$:
\begin{align*}
&\prod_{p\leq x}\left(1-\frac{(\alpha_f\alpha_g)(p)}{p}\right)^{-1}\left(1-\frac{(\alpha_f\beta_g)(p)}{p}\right)^{-1}\left(1-\frac{(\beta_f\alpha_g)(p)}{p}\right)^{-1}\times\\
&\times\left(1-\frac{(\beta_f\beta_g)(p)}{p}\right)^{-1}=
\begin{cases} 
A_{f\times f}e^\gamma(\log x)(1+O(e^{-C_{f, g}\sqrt{\log x}})), & f=g, \\
L_{f\times g}(1)+O(e^{-C_{f, g}\sqrt{\log x}}), & f\ne g,
\end{cases}
\end{align*}
and the prime number theorem for $L_{f\times g}(s)$:
\begin{align*}
\sum_{p^r\leq x}b_{f\times g}(n)\log n=\begin{cases} x+O(xe^{-c_{f, g}\sqrt{\log x}}), & f=g, \\ O(xe^{-c_{f, g}\sqrt{\log x}}), & f\ne g.\end{cases}
\end{align*}
Where $\alpha_j, \beta_j$ satisfy $(\alpha_j+\beta_j)(p)=\lambda_j(p)$, $(\alpha_j\beta_j)(p)=1$, $A_{f\times f}$ is the residue for $L_{f\times f}(s)$ at $s=1$, 
$c_{f,g}, C_{f,g}$ are positive constants such that $c_{f,g}<C_{f,g}$ depending on $f, g$, and $b_{f\times g}(n)$ are given by
\begin{align*}
b_{f\times g}(n)=\begin{cases}  (\alpha_f^r+\beta_f^r+\alpha_g^r+\beta_g^r)(p)/r, & n=p^r,\\ 0, & n\ne p^r. \end{cases}
\end{align*}
\end{cor}
In this corollary, we used the fact $b_{f\times g}(n)\geq 0$ when $f=g$, and the result that if $f\ne g$ then the condition (\ref{AZF}) are satisfied (see Ichihara \cite{ICI}). 
In this paper, we shall show Theorem \ref{MTH1} in Section \ref{PFTH1} and Theorem \ref{MTH2} in Section \ref{PFTH2}. 

\section{Proof of Theorem \ref{MTH1}}\label{PFTH1}
Let $F\in\mathcal{S}^{\#}$ and put the left hand of (\ref{MT1}) by $F_x(1)$. We shall give the approximate formula of $\log F_x(1)$. By using Remarks \ref{NEP} and (\ref{EP}), we can write
\begin{align}
\log F_x(1)=&\sum_{p\leq x}\sum_{r=1}^\infty\frac{b_F(p^r)}{p^r}\notag\\
=&\sum_{n\leq x}\frac{b_F(n)}{n}+\sum_{\sqrt{x}<p\leq x}\sum_{p^r>x}\frac{b_F(p^r)}{p^r}+\sum_{p\leq \sqrt{x}}\sum_{p^r>x}\frac{b_F(p^r)}{p^r}. \label{MLA}
\end{align}
It is clear that the second and third terms of (\ref{MLA}) are estimated as
\begin{align}
&\sum_{\sqrt{x}<p\leq x}\sum_{p^r>x}\frac{b_F(p^r)}{p^r}\ll\sum_{\sqrt{x}<p\leq x}\sum_{r=2}^\infty\frac{1}{p^r}\ll\sum_{\sqrt{x}<p\leq x}\frac{1}{p^2}\ll\frac{1}{\sqrt{x}},\label{LLA1}\\
&\sum_{p\leq \sqrt{x}}\sum_{p^r>x}\frac{b_F(p^r)}{p^r}\ll\sum_{p\leq \sqrt{x}}\sum_{r>\frac{\log x}{\log p}}\frac{1}{p^r}\ll\sum_{p\leq\sqrt{x}}\frac{1}{x}\ll\frac{1}{\sqrt{x}}\label{LLA2}.
\end{align}
Applying Perron's formula to the first term of (\ref{MLA}) (see Liu and Ye {\cite[Corollary 2.2]{L&Y}} or {\cite[Chapter 5.1]{M&V}}),
we get
\begin{align}
\sum_{n\leq x}\frac{b_F(n)}{n}=\frac{1}{2\pi i}\int_{b-iT}^{b+iT}\frac{x^s}{s}\log F(1+s)ds+O(e^{-c\sqrt{\log x}}), \label{LLB}
\end{align}
where we put $b=1/\log x$ and $T=e^{\sqrt{\log x}}$, and the following fact were used: 
\begin{align*}
\log F(\sigma)\ll\zeta(\sigma)\ll\frac{1}{\sigma-1} \;\; (\sigma\in\mathbb{R}_{>1}), \quad 
-\frac{\sqrt{\log x}}{2}+\log\log x\leq -c\sqrt{\log x}
\end{align*}
where $c\in(0,1/2)$ is a constant. 

Now we consider the integral of (\ref{LLB}). If we put $b'=c_F/\log T=c_F/\sqrt{\log x}$ and take large $x$, from (\ref{AZF}) we see that
\begin{align}
F(\sigma+it)\ne0; \quad t\in[-T,T], \ \sigma\geq1-b' \label{IZF}
\end{align}
on condition that $s=1$ is excluded when $F(s)$ has zero on $s=1$. Note that $F(s)$ has no Siegel zeros in the region $\text{Re }s\leq 1-b'$ because $x$ is taken large. 
Define the contour 
\begin{align*}
L_{+1}=&\{-b'+it\mid t\in[0,T]\}, & L_{-1}=&\{-b'+it\mid t\in[-T,0]\}, \\
L_{-2}=&\{-\sigma-iT\mid \sigma\in[-b',b]\}, & L_{+2}=&\{\sigma+iT\mid \sigma\in[-b',b]\},\\ 
C=&\{b'e^{i\theta}\mid \theta\in[-\pi,\pi]\}.&  &
\end{align*}
Using (\ref{IZF}) and Cauchy's residue theorem, we have
\begin{align}
\frac{1}{2\pi i}\int_{b-iT}^{b+iT}\frac{x^s}{s}\log F(s+1)ds
=\frac{1}{2\pi i}\int_{C}\frac{x^s}{s}\log F(s+1)ds+O\Biggl(\sum_{j=\pm1,\pm2}I_j\Biggr), \label{LLC}
\end{align}
where
\begin{align*}
I_j=\frac{1}{2\pi i}\int_{L_j}\frac{x^s}{s}\log F(s+1)ds.
\end{align*}

First we shall calculate estimates of $I_j$. In the case of $j=\pm2$, Phragm\'{e}n-Lindel\"{o}f theorem and Stirling's formula give that $\log F(1-\sigma\pm iT)\ll(\log T)^2$ for $\sigma\in[-b',b]$. Then $I_{\pm2}$ are estimated as
\begin{align}
I_{\pm2}
\ll\frac{(\log T)^2}{T}\int_{-b'}^{b}x^{\sigma}d\sigma\ll (\sqrt{\log x})^2e^{-\sqrt{\log x}}\ll e^{-C_F\sqrt{\log x}}. \label{LLK}
\end{align}
In the case of $j=\pm1$, we use the following result (see \cite[Lemma 6.3]{M&V}):
\begin{lem}\label{FZE}
Let $f(z)$ be an analytic function in the region containing the disc $|z|\leq 1$, supposing $|f(z)|\leq M$ for $|z|\leq 1$ and $f(0)\ne 0$. Fix $r$ and $R$ such that $0<r<R<1$. Then, for $|z|\leq r$ we have
\begin{align*}
\frac{f'}{f}(z)=\sum_{|\rho|\leq R}\frac{1}{z-\rho}+O\left(\log\frac{M}{|f(0)|}\right)
\end{align*}
where $\rho$ is a zero of $f(s)$.
\end{lem}
Put $f(z)=(z+1/2+it)^mF(1+z+(1/2+it))$, $R=5/6$, $r=2/3$ in Lemma \ref{FZE}, and use the assuming the condition (\ref{AZF}). Then the following estimates are obtained by the same discussion of the proof of \cite[Theorem 6.7]{M&V}:   
\begin{align*}
\log s^mF(s+1)\ll
\begin{cases}
\log(|t|+4), & |t|\geq 7/8 \text{ and } \sigma\geq -b',\\
 1, & |t|\leq 7/8 \text{ and } \sigma\geq -b',
\end{cases}
\end{align*}
and $I_{\pm 1}$ are estimated as
\begin{align}
I_{\pm1}\ll&\left(\int_0^{7/8}+\int_{7/8}^T\right)\frac{x^{-b'}}{|s|}\left(|\log s^m|+|\log s^mF(s+1)|\right)dt\notag\\
\ll&\int_0^{7/8}\frac{x^{-b'}}{b'}(\log b'+1)dt+\int_{7/8}^T\frac{x^{-b'}}{t}(\log(t+b')+\log(t+4))dt\notag\\
\ll&e^{-c_F\sqrt{\log x}}\sqrt{\log x}\log\log x+e^{-c_F\sqrt{\log x}}\log x
\ll e^{-C_F\sqrt{\log x}}.\label{LLM}
\end{align}

Secondly we consider the integral term of (\ref{LLC}). From (\ref{AC}), we see that $F(s)$ has a pole of order $m$ on $s=1$ where $m\in\mathbb{Z}_{\geq 1}$, or has a zero of order $-m$ in $s=1$ where $m\in\mathbb{Z}_{\leq 0}$. Considering the Laurent expansion of $F(s)$ in $s=1$, we get $c_{-m}=\lim_{s\to1}(s-1)^mF(s)\ne0$ for $m\in\mathbb{Z}$. Therefore, the following formula is obtained by Cauchy's residue theorem:
\begin{align}
\frac{1}{2\pi i}\int_{C}\frac{x^s}{s}\log F(s+1)ds=-\frac{m}{2\pi i}\int_C\frac{x^s}{s}\log sds+\log c_{-m} \label{LLD}
\end{align}
Here, the first term of (\ref{LLD}) is written as 
\begin{align}
\int_C\frac{x^s}{s}\log sds=i(\log b')\int_{-\pi}^\pi e^{b'e^{i\theta}\log x}d\theta-\int_{-\pi}^\pi\theta e^{b'e^{i\theta}\log x}d\theta. \label{LLE}
\end{align}
Using termwise integration, the first and second terms on the right hand side of (\ref{LLE}) are calculated as
\begin{align}
\int_{-\pi}^\pi e^{b'e^{i\theta}\log x}d\theta=&\int_{-\pi}^\pi d\theta+\sum_{r=1}^\infty\frac{(b'\log x)^r}{r!}\int_{-\pi}^\pi e^{ir\theta}d\theta=2\pi,\label{LLF}\\
\int_{-\pi}^\pi\theta e^{b'e^{i\theta}\log x}d\theta
=&\sum_{r=1}^\infty\frac{(b'\log x)^r}{r!}\int_{-\pi}^\pi\theta e^{ir\theta}d\theta
=\frac{2\pi}{i}\sum_{r=1}^\infty\frac{(-1)^r}{r!}\int_{0}^{b'\log x}u^{r-1}du\notag\\=&
\frac{2\pi}{i}\int_{0}^{b'\log x}\frac{e^{-u}-1}{u}du. \label{LLG}
\end{align}
Moreover, (\ref{LLG}) is calculated as
\begin{align}
\int_{0}^{b'\log x}\frac{e^{-u}-1}{u}du=&\gamma+\int_1^{b'\log x}\frac{du}{u}-\int_{b'\log x}^\infty\frac{e^{-u}}{u}du\notag\\
=&\gamma+\log\log x+\log b'+O(e^{-C_F\sqrt{\log x}}), \label{LLH}
\end{align}
where the following result was used:
\begin{align*}
\int_0^1\frac{1-e^{-u}}{u}du-\int_1^\infty\frac{e^{-u}}{u}du=\gamma.
\end{align*}

Finally combining (\ref{MLA})--(\ref{LLH}), we get
\begin{align}
\log F_x(1)=\log c_{-m}+m\gamma+m\log\log x+O(e^{-C_F\sqrt{\log x}}). \label{MFX}
\end{align}
Taking exponentials in both sides of (\ref{MFX}) and using the fact $e^y=1+O(y)$ for $y\ll 1$, we complete 
the proof of Theorem \ref{MTH1}.

\section{Proof of Theorem \ref{MTH2}}\label{PFTH2}
First we  shall show the following result from Mertens' 3rd theorem:
\begin{pro}[Mertens' 2nd theorem]\label{M2M}
Let $F\in\mathcal{S}^{\#}$ and assume (\ref{AEP}) and (\ref{AZF}). Then we have
\begin{align*}
\sum_{p\leq x}\frac{b_F(p)}{p}=m\log\log x+M+O(e^{-C_F\sqrt{\log x}}),
\end{align*}
where $M$ is a constant given by
\begin{align*}
M=\log c_{-m}+m\gamma-\sum_{p\text{:prime}}\sum_{r=2}^\infty\frac{b_F(p^r)}{p^r}.
\end{align*}
\end{pro}
Namely we may call the constant $M$ \textit{generalized Mertens' constant}.
\begin{proof}
By (\ref{MLA}) the sum of statement of Proposition \ref{M2M} is written as
\begin{align}
\sum_{p\leq x}\frac{b_F(p)}{p}=&\log F_x(1)-\sum_{p\leq x}\sum_{r=2}^\infty\frac{b_F(p^r)}{p^r}\notag\\
=&\log F_x(1)-\sum_{p\text{:prime}}\sum_{r=2}^\infty\frac{b_F(p^r)}{p^r}+\sum_{p>x}\sum_{r=2}^\infty\frac{b_F(p^r)}{p^r}. \label{M2A}
\end{align}
By the trivial estimate, the third term of the right-hand side of (\ref{M2A}) is estimated as
\begin{align}
\sum_{p>x}\sum_{r=2}^\infty\frac{b_F(p^r)}{p^r}\ll\sum_{p>x}\frac{1}{p^2}\ll\frac{1}{x}. \label{M2C}
\end{align}
Therefore from (\ref{MFX})--(\ref{M2C}), Proposition \ref{M2M} is obtained.
\end{proof}

Secondly we shall show the following formula from Proposition \ref{M2M}:
\begin{pro}[Mertens' 1st theorem]\label{M1M}
Let $F\in\mathcal{S}^{\#}$ and assume (\ref{AEP}) and (\ref{AZF}). Then we have
\begin{align*}
\sum_{p\leq x}\frac{b_F(p)\log p}{p}=m\log x+M_1+O(e^{-C_F'\sqrt{\log x}}),
\end{align*}
where $C_F'$ is a positive constant smaller than $C_F$ on Theorem \ref{MTH1}, $M_1$  is a constant given by
\begin{align*}
M_1=-\int_2^\infty\frac{\Delta_{2F}(u)}{u}du+M\log 2+m(\log 2)(\log\log 2-1),
\end{align*}
and $\Delta_{2F}(u)$ is given by
\begin{align*}
\Delta_{2F}(u)=\sum_{p\leq u}\frac{b_F(p)}{p}-m\log\log u-M,
\end{align*}
which is estimated as $\Delta_{2F}(u)=O(e^{-C_f\sqrt{\log u}})$.
\end{pro}
\begin{proof}
Using partial summation formula, we have
\begin{align}
&\sum_{p\leq x}\frac{b_F(p)\log p}{p}\notag\\
&=(\log x)\sum_{p\leq x}\frac{b_F(p)}{p}-\int_2^x\frac{1}{u}\sum_{p\leq u}\frac{b_F(p)}{p}du\notag\\
&=(\log x)\sum_{p\leq x}\frac{b_F(p)}{p}-\int_2^x\frac{m\log\log u+M}{u}du-\int_2^x\frac{\Delta_{2F}(u)}{u}du\notag\\
&=S_1+S_2+S_3. \label{M1A}
\end{align}
From Proposition \ref{M2M}, we see that $\Delta_{2F}(u)=O(e^{-C_F\sqrt{\log x}})$ and 
\begin{align}
S_1=&m(\log x)\log\log x+M\log x+O(e^{-C_F'\sqrt{\log x}}), \label{M1B}\\
S_2=&-m(\log x)\log\log x-M\log x+m\log x+m(\log 2)\log\log 2-\notag\\&-m\log 2+M\log 2, \label{M1C}\\
S_3=&-\int_2^\infty\frac{\Delta_{2F}(u)}{u}du+O(e^{-C_F'\sqrt{\log x}}). \label{M1D}
\end{align}
Combining (\ref{M1A})--(\ref{M1D}), we complete the proof of  Proposition \ref{M1M}.
\end{proof}

Finally we shall show Theorem \ref{MTH2} from Proposition \ref{M1M}. The left hand side of Theorem \ref{MTH2} is written as follows:
\begin{align}
\sum_{n\leq x}b_F(n)\log n=\sum_{p\leq x}b_F(p)\log p+\sum_{p^r\leq x, \ r\geq 2}b_F(p^r)\log p^r. \label{MPA}
\end{align}
The second term on right-hand side of (\ref{MPA}) is estimated as
\begin{align}
\sum_{p^r\leq x, \ r\geq 2}b_F(p^r)\log p^r\ll\sum_{p\leq\sqrt{x}}\sum_{r\leq\frac{\log x}{\log p}}\log p^r\ll\sqrt{x}(\log x)^2. \label{MPB}
\end{align}
Applying partial summation to the first term of the right-hand of (\ref{MPA}), we have
\begin{align}
\sum_{p\leq x}b_F(p)\log p
&=x\sum_{p\leq x}\frac{b_F(p)\log p}{p}-\int_2^x\sum_{p\leq u}\frac{b_F(p)\log p}{p}du \notag\\
&=x\sum_{p\leq x}\frac{b_F(p)\log p}{p}-\int_2^x(m\log u+M_1)du-\int_2^x\Delta_{1F}(u)du\notag\\
&=:T_1+T_2+T_3 \label{MPC}
\end{align}
where$\Delta_{1F}(u)$ is given by
\begin{align*}
\Delta_{1F}(u)=\sum_{p\leq u}\frac{b_F(p)\log p}{p}-m\log u-M_1.
\end{align*}
Proposition \ref{M1M} gives that $\Delta_{1F}(u)=O(ue^{-C_F'\sqrt{\log u}})$ and 
\begin{align}
&T_1=mx\log x+M_1x+O(xe^{-C'_F\sqrt{\log x}}),\label{MPD}\\
&T_2=-mx\log x+mx-M_1x+2(m\log2+M_1),\\
&T_3\ll\left(\int_2^{\sqrt{x}}+\int_{\sqrt{x}}^x\right)e^{-C_F'\sqrt{\log u}}du
\ll \sqrt{x}+x e^{-C_F'\sqrt{\log\sqrt{x}}}\ll xe^{-C_F''\sqrt{\log u}}. \label{MPE}
\end{align} 
Therefore from (\ref{MPA})--(\ref{MPE}), the proof of Theorem \ref{MTH2} is completed.


\bibliographystyle{alpha}

\begin{thebibliography}{10}
\bibitem{DEL} P. Deligne, \textit{La conjecture de Weil. I}, Publ. Math. Inst. Hautes \'{E}tudes Sci. \textbf{43} (1974), 273--307.
\bibitem{ICI} Y. Ichihara, \textit{The Siegel-Walflsz theorem for Rankin-Selberg $L$-functions associated with two cusp forms}, Acta Arith. (3) \textbf{92} (2000), 215--227.
\bibitem{KP1} J. Kaczorowski and A. Perelli, \textit{On the prime number theorem for the Selberg class}, Arch. Math. \textbf{80} (2003) 255--263.
\bibitem{LAN} E. Landau, \textit{Einf\"{u}hrung in die Elementare und Analytische Theorie der Algebraischen Zahlen und der Ideale}, second edition, Chelsea Publishing Co., New York, 1949.
\bibitem{L&Y} J. Liu and Y. Ye, \textit{Perron's formula and the prime number theorem for automorphic L-functions}, Pure Appl. Math. Quart \textbf{3} (2007), 481--497.
\bibitem{MER} F. Mertens, \textit{Ein Beitrag zur analytischen Zahlentheorie}, J. Reine Angew. Math. \textbf{78} (1874), 46--62.
\bibitem{M&V} H. L. Montgomery and R. C. Vaughan, \textit{Multiplicative Number Theory. I. Classical Theory}, Cambridge University Press, 2007.
\bibitem{MOT} Y. Motohashi, \textit{Analytic Number Theory I}, Asakura Pulishing, Tokyo, 2009, (in Japanese).
\bibitem{ROS} M. Rosen, \textit{A generalization of Mertens' theorem}, J. Ramanujan Math. Soc. (1) \textbf{14} (1999), 1--19
\bibitem{SE2} A. Selberg, \textit{Old and new conjectures and results about a class of Dirichlet series}, In Proc. Amalfi Conf.
Analytic Number Theory, E. Bombieri et al. eds., 367--385, Universit\`{a}di Salerno 1992; Collected Papers, Vol. II, 47--63, Berlin-Heidelberg-New York 1991.
\end{thebibliography}

\end{document}